\documentclass[12pt]{article}         

 \usepackage[T1]{fontenc}              
 \usepackage[latin1]{inputenc}         
      \input{dehyphtex}                
 \usepackage[width=16cm, height=22cm]{geometry}   
 \usepackage{amsmath,amssymb}          
\usepackage{hyperref} 
 \usepackage[amsmath,thmmarks,hyperref]{ntheorem} 
 \usepackage{icomma}                   
 \usepackage{enumerate}                
 \usepackage{color}                    
 \usepackage{setspace}                 
 \usepackage{needspace}                
 \usepackage{url}                      
 \usepackage{mathrsfs}                 
 \usepackage{slashed}

 \numberwithin{equation}{section}
 
 \usepackage[numbers,square]{natbib}
\theoremstyle{nonumberplain}  
\theoremheaderfont{\itshape}  
\theorembodyfont{\normalfont}  
\theoremseparator{.}  
\theoremsymbol{$\Box$}
\newtheorem{proof}{Proof} 
  
\theoremstyle{plain}  
\theoremheaderfont{\normalfont\bfseries}  
\theorembodyfont{\itshape}  
\theoremsymbol{~}
\theoremseparator{.}  
\newtheorem{prop}{Proposition}[section]  
\newtheorem{corollary}[prop]{Corollary}  
\newtheorem{lemma}[prop]{Lemma}  
\newtheorem{theorem}[prop]{Theorem}   

\theorembodyfont{\normalfont}  
 
\newtheorem{remark}[prop]{Remark}
\newtheorem{example}[prop]{Example}  
  
\newtheorem{definition}[prop]{Definition} 
\newtheorem{notation}[prop]{Notation} 

\theoremstyle{nonumberplain}
\theorembodyfont{\normalfont}

\newcommand*{\grad}{\operatorname{grad}}

\newcommand{\R}{\mathbb{R}}
\newcommand{\Eh}{\mathcal{E}}

\newcommand{\N}{\mathbb{N}}

\newcommand{\C}{\mathbb{C}}
\newcommand{\dd}{\mathrm{d}}

\newcommand{\End}{\mathrm{End}}

\renewcommand*{\div}{\operatorname{div}}

\newcommand{\id}{\mathrm{id}}
\newcommand{\Cinf}{C^\infty}

\newcommand{\m}{\mathfrak{m}}
\newcommand{\DD}{\mathrm{D}}
\newcommand{\V}{\mathcal{V}}
\newcommand{\im}{\mathrm{im}}
\renewcommand{\Re}{\mathrm{Re}}
\newcommand{\X}{\mathfrak{X}}

\usepackage{tikz}
\newcommand*\circled[1]{\tikz[baseline=(char.base)]{
            \node[shape=circle,draw,inner sep=2pt] (char) {#1};}}

\title{Vector Fields with a non-degenerate Source}
\author{Matthias Ludewig}

\begin{document}

\maketitle

\begin{center}
  Universität Potsdam / Institut für Mathematik \\ 
  Am Neuen Palais 10 / 14469 Potsdam, Germany \\ \medskip
 matthias.ludewig@uni-potsdam.de 
\end{center}

\vspace{0.5cm}

\begin{abstract}
\noindent We discuss the solution theory of operators of the form $\nabla_X + A$, acting on smooth sections of a vector bundle with connection $\nabla$ over a manifold $M$, where $X$ is a vector field having a critical point with positive linearization at some point $p \in M$. As an operator on a suitable space of smooth sections $\Gamma^\infty(U, \V)$, it fulfills a Fredholm alternative, and the same is true for the adjoint operator. Furthermore, we show that the solutions depend smoothly on the data $\nabla$, $X$ and $A$.
\end{abstract}


\section{Introduction}

Let $M$ be a manifold of dimension $n$ and $X$ be a smooth vector field on $M$. We consider points $p \in M$, where the vector field has what we will call a {\em strictly positive source}, meaning that $X(p) = 0$ and all eigenvalues of the linearization of $X$ at $p$  have strictly positive real part.

Let furthermore $\V$ be a real or complex vector bundle over $M$ endowed with a connection $\nabla$ and some given endomorphism field $A$, i.e.\ a smooth section of the bundle $\End(\V)$. In this paper, we discuss properties of the differential operator $\nabla_X + A$, where we assume that $X$ has a strictly positive source at $p \in M$, as explained above. The goal is to solve differential equations of the form
\begin{equation} \label{TheEquation0}
  (\nabla_X + A)u = \lambda u + v.
\end{equation}
In the case that $M \subseteq \R^n$ is open and $A$ is just a matrix-valued function, we may assume that $p = 0$ and the eigenvalue equation $(\nabla_X+A)u=v$ is equivalent to the system
\begin{equation}
  X^i(y) \frac{\partial}{\partial y^i}u^k(y) + A^k_j(y) u^j(y) = v^k(y), ~~~~~~~ k = 1, \dots, m
\end{equation}
of scalar first order equations, where the functions $X^i$ vanish at zero and the real part of each of the eigenvalues of the matrix $(\DD_j X^i)_{ij}$ is positive.

Usually, first order equations can be easily solved with the method of characteristics, but the singular nature of the operator does not admit this approach near the critical point of $X$. In fact, we will see that operators of this type have some striking analogies to elliptic operators;
we show in particular that as an operator on a suitable space of smooth sections $\Gamma^\infty(U, \V)$, the operator fulfills a Fredholm alternative, and the same is true for the dual operator on the space of distributions $\Eh^\prime(U, \V^*)$. Furthermore, the solutions depend smoothly on the data and the solution operator is a smooth map between suitable Fréchet spaces.

Operators of the above form appear in various situations in geometry and mathematical physics. For example, the so-called recursive transport equations that appear in the construction of the asymptotic expansion of the heat kernel on a Riemannian manifold or the Hadamard solution associated to a d'Alembert operator on a Lorentzian manifold are of the form \eqref{TheEquation0}. Also, the transport equations one has to solve in semiclassical analysis when constructing formal WKB expansions to Schrödinger operators $\hbar^2 \Delta + V$ near critical points of the potential $V$ have the form \eqref{TheEquation0}. We discuss these examples in section \ref{section4}.

The study of differential equations of this form has some history in the theory of WKB approximations in semiclassical analysis (see e.g.\ \cite{helffersjostrand84}, \cite{helffer88} and \cite{dimassisjoestrand99}). Most of the results of this paper are therefore not particularly new but presented in a more are general and conceptual form, involving in particular a vector-valued geometric setting. In particular the formulation of Thm.\ \eqref{ThmFredholm} seems more clear and less {\em ad hoc} than the corresponding statements available in the literature. The use of the more refined estimates from Thm.\ \ref{ThmODESolutionEstimate} to get smooth dependence on initial data (Thm.\ \ref{SmoothDependence}) seems to be new.

 We use many of the ideas from the cited references, but we will need to adapt the proofs to fit the more general setting we shall discuss.

\medskip

{\bfseries Acknowledgement.} Thanks to Florian Hanisch, Christoph Stephan, Christian Bär and Rafe Mazzeo for many helpful discussions. I am also indebted to Robert Bryant and Alberto Abbondandolo, who both provided great help via mathoverflow.net. Last but not least, I want to thank Potsdam Graduate School (POGS) for financial support.


\section{Outline of the Results} \label{section1_2}

\begin{definition}
Let $M$ be a manifold of dimension $n$ and $X$ be a smooth vector field on $M$. A point $p \in M$  with $X(p)=0$ is called {\em strictly positive source} of $X$ if the real parts of all eigenvalues of the linearization $\nabla X|_p \in \End(T_pM)$ are strictly positive.
\end{definition}

Whenever we speak of eigenvalues of an endomorphism of a finite-dimensional vector space, we always mean the roots of the characteristic polynomial, counted with algebraic multiplicity. In the definition above, $\nabla$ is any connection on $TM$; because $X(p) = 0$, the linearization $\nabla X|_p$ is independent of the choice of connection.

\begin{definition} \label{DefStarshaped}
If $X$ has a strictly positive source at $p \in M$, an open neighborhood $U$ of $p$ is called {\em star-shaped} around $p$ with respect to $X$, if for all $q \in U$, the flow $\Phi_t(q)$ of $X$ exists for all $t\leq 0$ with
\begin{equation*}
  \lim_{t \longrightarrow - \infty} \Phi_t(q) = p,
\end{equation*}
and furthermore $\Phi_t(U) \subseteq U$ for all $t\leq 0$. 
The stable manifold theorem (see for example \citep[p.\ 116]{perko91}) guarantees the existence of star-shaped neighborhoods around $p$.
\end{definition}

For open subsets $U \subseteq M$, we equip the space of sections $\Gamma^\infty(U, \V)$ with its Fréchet topology (induced by the $C^m$ norms on compact subsets of $U$). The dual space is then $\Eh^\prime(U, \V^*)$, the space of compactly supported distributions with values in $\V^*$. 

\begin{theorem}[Fredholm Alternative] \label{ThmFredholm}
Let $X$ be a vector field on $M$ with a strictly positive source at $p$ and let $U$ be star-shaped around $p$ with respect to $X$. Consider the operator $\nabla_X + A$ as a bounded linear operator on $\Gamma^\infty(U, \V)$, as well as the dual operator $(\nabla_X + A)^\prime$ on the space $\Eh^\prime(U, \V^*)$.

Then either
\begin{enumerate}
\item[a)] $\lambda$ is not an eigenvalue of $\nabla_X + A$ and the inhomogeneous equation
\begin{equation} \label{InhomogeneousEquation}
 (\nabla_X + A)u = \lambda u + v
\end{equation} 
has exactly one solution for each $v \in \Gamma^\infty(U, \V)$; or
\item[b)] both the homogeneous equation
\begin{equation} \label{HomogeneousEquation}
  (\nabla_X + A)u = \lambda u, ~~~~~~~ u \in \Gamma^\infty(U, \V)
\end{equation}
and the dual equation
\begin{equation} \label{DualEquation}
  (\nabla_X + A)^\prime T = \lambda T, ~~~~~~~ T \in \Eh^\prime(U, \V^*)
\end{equation}
have $k < \infty$ linearly independent solutions. Then the inhomogeneous equation \eqref{InhomogeneousEquation} has a solution if and only if $v \in \ker ((\nabla_X + A)^\prime- \lambda)_\perp$. In this case, the space of solutions is $k$-dimensional affine subspace of $\Gamma^\infty(U, \V)$, the direction of which is the space of solutions to \eqref{HomogeneousEquation}.
\end{enumerate}
\end{theorem}

 In the theorem,
\begin{equation*}
  \ker ((\nabla_X + A)^\prime- \lambda)_\perp := \{ u \in \Gamma^\infty(U, \V) \mid T(u) = 0 ~~ \forall ~~ T \in \ker ((\nabla_X + A)^\prime - \lambda) \}.
\end{equation*}
We will give a proof in Section \ref{section3_5}.
A similar theorem holds for the dual operator $(\nabla_X + A)^\prime$, compare Cor. \ref{ThmFredholmDual}.  Results similar to Thm.\ \ref{ThmFredholm} appear in \citep{dimassisjoestrand99} and partly in \citep[Thm.\ 2.3.1]{helffer88}, but in a different form. 

In this sense, operators of the form $\nabla_X + A$ behave similar to elliptic operators. There is no analog to elliptic regularity however, and indeed there may be additional non-smooth solutions of \eqref{HomogeneousEquation} (see Example \ref{ExampleLowRegularity}). For this reason, it is suitable to consider the spaces of smooth sections as opposed to some Banach or Hilbert space setting one usually considers.

\begin{theorem}[Eigenvalues]\label{ThmEigenvalues}
Assume that $\V$ is a complex vector bundle of rank $m$. A number $\lambda \in \C$ is an eigenvalue of the operator $\nabla_X + A$ (i.e.\ the homogeneous equation \eqref{HomogeneousEquation} has a solution), if and only if it has the form
\begin{equation} \label{TheEigenvalue}
  \lambda = \alpha_1 \mu_1 + \dots + \alpha_n \mu_n + \rho_j, 
\end{equation}
for some multi-index $\alpha \in \N_0^n$ and a number $1 \leq j \leq m$, where $\mu_1, \dots \mu_n$ are the eigenvalues of the linearization $\nabla X|_p$ of $X$ and $\rho_1, \dots, \rho_m$ are the eigenvalues of $A(p)$, each repeated with multiplicity. Furthermore, in this case we have the inequality
\begin{equation*}
  1 \leq \dim \ker (\nabla_X + A) \leq m(\lambda),
\end{equation*}
where $m(\lambda)$ is the the number of ways to write $\lambda$ in the above form for different eigenvalues $\rho$ of $A(p)$ and multi-indices $\alpha$.
\end{theorem} 

\begin{remark} \label{RemarkOnMlambda}
In general, there need not be $m(\lambda)$ eigenfunctions to a given eigenvalue $\lambda$. This "lack of eigenfunctions" can have two reasons: First, $\DD X|_0$ or $A(p)$ need not be diagonalizable, i.e.\ a number $\rho$ or $\mu_j$ may be a root of order $k$ of the characteristic polynomial of $A(p)$ or $\nabla X|_p$ at $p$, but only admit $l<k$ linearly independent eigenfunctions. In this case, we "loose" eigenfunctions of $\nabla_X + A$. However, even if both these endomorphisms are diagonalizable, there is no guarantee to have $m(\lambda)$ eigenfunctions, unless the vector field fulfills some additional conditions, which are discussed in section \ref{section4}.
\end{remark}

Thm.\ \ref{ThmEigenvalues} is proved in Section \ref{section3}. In Section \ref{sectionsmooth}, we furthermore show that the solutions to \eqref{InhomogeneousEquation} depend smoothly on the data $X$, $A$, $v$ and the connection $\nabla$ (at least in the case that $\nabla_X + A$ has trivial kernel). For this, we need some estimates on the flow of the vector field $X$, which are established in Section \ref{section1_5}. Parts of these estimates are needed in the other sections as well. In Section \ref{section4}, we apply the results obtained so far to some particular problems and give further discussion.

\begin{remark}
Clearly, all results can as well be applied to the case that a vector field $X$ singular at a point $p \in M$ has a negative definite linearization $\nabla X|_p$. One only needs to replace $X$ by $-X$. Furthermore, if $X$ is a vector field with a hyperbolic critical point at $p$, one can apply the theorems above to both the stable and the unstable manifold corresponding to the singularity.
\end{remark}


\section{Estimates} \label{section1_5}

In this section, we establish some needed estimates on the flow of vector fields and the solutions to linear equations. All these results should be well-known; however, no definite references seem to exist quite in the form needed. We divested the technical lemmas to the appendix in order keep the section brief.

\begin{notation} \label{NotationXU}
Let $U \subseteq M$ be an open subset. We denote by $\X_U$ set of all vector fields $X$ that have a strictly positive source at some point $p \in U$ (that may depend on $X$) such that $U$ is star-shaped around $p$ with respect to $X$. For $X \in \X_U$, the flow $\Phi_t(y) =\Phi_t^X(y)$ exists for all $t \leq 0$ and $y \in U$, and depends smoothly on $X$ (for $X$ in the interior of $\X_U$). 
\end{notation}

Let $U \subseteq \R^n$ be open and let $V$ be a finite-dimensional real or complex vector space (we assume that we have chosen a norm, but the results do not really depend on this choice, as all norms are equivalent). Choose $A \in C^\infty(U, \End(V))$ and $X \in \X_U$. The parametrized initial value problem
  \begin{equation} \label{InitialValueProblem23}
    \frac{\partial}{\partial t}E_{X,A}(t, y) = A(\Phi_t^X(y))E_{X,A}(t, y), ~~~~~~~ E_{X, A}(0, y) = \id
  \end{equation}
has a unique solution $E_{X, A} \in C^\infty(\R \times U, \End(V))$, defined for all $t\leq 0$ and $y \in U$. $E_{X, A}$ depends smoothly on $A$, $X$ and $y$, $t$.

\begin{theorem} \label{ThmODESolutionEstimate}
Let $X \in \X_U$ be a vector field with a strictly positive source $p \in U$, and let $A \in C^\infty(U, \End(V))$. Choose $\nu \in \R$ with 
\begin{equation*} 
  \nu < \inf \Re \, \mathrm{spec}\, A(p).
\end{equation*}
Then for every $m \in \N$ and every compact neighborhood $K$ of $p$, there exists a $C^m(K)$-ball
\begin{equation*} \label{mBall}
  B_R^{m,K}(X, A) := \Bigl\{(X^\prime, A^\prime) ~\bigl|~ \|X^\prime - X\|_{C^m(K)}^2 + \|A^\prime - A\|_{C^m(K)}^2 \leq R^2 \Bigr\}
\end{equation*}
and a constant $C>0$ such that
\begin{equation}
  \| E_{X^\prime, A^\prime}(t) \|_{C^m(K)} \leq C e^{ t \nu}
\end{equation}
for all $t \leq 0$ and for each $(X^\prime, A^\prime) \in B_R^{m,K}(X, A) \cap \bigl(\X_U \times C^\infty(U, \End(V))\bigr)$.
\end{theorem}

\begin{proof}
By differentiating \eqref{InitialValueProblem23} with respect to $y$, we see that $D^\alpha E_{X, A}(t, y)$ fulfills the differential equation
\begin{equation*}
  \frac{\partial}{\partial t} D^\alpha E_{X, A}(t, y) = \sum_{\beta \leq \alpha} \binom{\alpha}{\beta} \DD^{\alpha - \beta} [A(\Phi_t^X(y))] \,\, \DD^{\beta} E_{X, A}(t, y), ~~~~~~ D^\alpha E_{X, A}(0, y) = 0.
\end{equation*}
Set
\begin{equation*}
  W_{X, A}(t, y) := \left( \binom{\gamma}{\beta} D^{\gamma-\beta} [A(\Phi_t^X(y))] \right)_{|\gamma|, |\beta| \leq m}
\end{equation*}
and let $F_{X,A}(t, y)$ be the solution of the problem
\begin{equation*}
  \frac{\partial}{\partial t} F_{X, A}(t, y) = W_{X,A}(t, y) F_{X,A}(t, y), ~~~~~~~~ F_{X,A}(0, y) = \mathrm{id}.
\end{equation*}
By construction, $\|E(t)\|_{C^m(K)} \leq C \| F(t)\|_{C^0(K)}$ for all $t \leq 0$ and some $C$ depending on the choice of matrix norm and on $m$, so estimating the $C^0$ norm of $F$ amounts to estimating the $C^m$ norm of $E$.
  
There exists a basis such that $W_{X,A}(t, y)$ is an upper triangular matrix with copies of $A(\Phi_t^X(y))$ on the diagonal, so $W_{X,A}(t, y)$ and $A(\Phi_t^X(y))$ have the same eigenvalues. Furthermore, it is clear that
\begin{equation*}
  W_0 := \lim_{t \longrightarrow -\infty} W_{X,A}(t, y)
\end{equation*}
exists and is independent of $y$. 
 
The rest is an application of Thm.\ \ref{ThmEstimateLemma}. Set $\varepsilon = \min \Re \, \mathrm{spec}(W_0) - \nu$ and choose $t_0< 0$ such that
\begin{equation*}
  \|W_{X,A}(t, y) - W_0\| < \frac{\varepsilon/2}{M(W_0, \varepsilon/2)} - \delta
\end{equation*}
for all $t<t_0$ and all $y \in K$, where $\delta>0$ is chosen so small that the right side is still positive. Then for all $(X^\prime, A^\prime)$ such that
\begin{equation} \label{Distance}
\|W_{X,A}(t, y) - W_{X^\prime,A^\prime}(t, y)\|< \delta
\end{equation} 
for all $t\leq 0$ and all $y \in K$, we get that
\begin{equation*}
  \|W_{X^\prime,A^\prime}(t, y) - W_0\| < \frac{\varepsilon/2}{M(W_0, \varepsilon/2)}
\end{equation*}
by the triangle inequality. By Thm.\ \ref{ThmEstimateLemma}, this means that
\begin{equation*}
  \|F_{X^\prime,A^\prime}(t, y)\| \leq C e^{t\nu},
\end{equation*}
for all $y \in K$, where 
\begin{align*}
  C &= 
  M(W_0, \varepsilon/2) M(W_0, \varepsilon) \exp\left(-t_0M(W_0, \varepsilon)\sup_{t \leq 0}\|W_{X^\prime,A^\prime}(t, y) - W_0\|\right) \\
  &\leq M(W_0, \varepsilon/2) M(W_0, \varepsilon) \exp\left(-t_0M(W_0, \varepsilon)\Bigl(\delta + \sup_{t \leq 0, y \in K}\|W_{X,A}(t, y) - W_0\|\Bigr)\right).
\end{align*}
On the other hand, \eqref{Distance} holds certainly in some $C^m(K)$-ball around $(X, A)$ small enough, as $W_{X, A}$ depends $C^m$-continuously on $X$ and $A$.
\end{proof}

\begin{corollary}[Flow Estimates] \label{FlowEstimatesII}
Let $X \in \X_U$ have a positive definite source at $p \in U$. Assume that
\begin{equation*}
  \nu < \inf \mathrm{Re} \,\mathrm{spec}(\DD X|_p).
\end{equation*}
Then for every $m \in \N_0$ and every compact subset $K$ of $U$, there exists an open $C^m(K)$-ball $B^{m,K}_R(X)$  in $\X_U$ around $X$ and a constant $C$ such that
\begin{equation}
  \| \DD \Phi_t^{X^\prime} \|_{C^m(K)} \leq Ce^{\nu t}
\end{equation}
for all $t \leq 0$ and $X^\prime \in B^{m,K}_R(X)$.
\end{corollary}

\begin{proof}
 By differentiating the defining equation $\frac{\partial}{\partial t}{\Phi}_t(y) = X(\Phi_t(y))$, $\Phi_0(y) = y$ for the flow with respect to $y$, we get that $\DD \Phi_t(y)$ fulfills the differential equation
\begin{equation*}
  \frac{\partial}{\partial t} \DD \Phi_t|_y = \DD X|_{\Phi_t(y)} \cdot \DD \Phi_t|_y, ~~~~~~ \DD \Phi_0|_y = \mathrm{id}
\end{equation*}
Now apply Thm.\ \ref{ThmODESolutionEstimate} with $A := \DD X$.
\end{proof}


\section{Flat Solutions} \label{section2}

In this section, we prove the following theorem, which is a key to everything that follows. Recall that a local section $u$ of $\V$ is called flat at $p$ if $U(p)=0$ and all its derivatives (repeated, into any direction) vanish at $p$ as well. 

\begin{theorem} \label{ThmFlat}
Let $X$ be a smooth vector field on $M$ with a strictly positive source at $p \in M$ and let $U$ be open and star-shaped around $p$ with respect to $X$. Then for each section $v \in \Gamma^\infty(U, \V)$ that is flat at $p$, there exists a unique flat section $u \in \Gamma^\infty(U, \V)$ such that
\begin{equation} \label{TheEquationFlat}
  (\nabla_X + A)u = v.
\end{equation}
In particular, the homogeneous equation \eqref{HomogeneousEquation} has no non-trivial flat solutions.
\end{theorem}

This theorem appears in \citep[Prop.\ 2.3.7]{helffer88} and \citep[p.\ 23]{dimassisjoestrand99} (and probably other sources), though in a less general setup. However, the proof of it is held very brief in all sources known to the author, which made it seem worthwhile to write down a more extensive version.

To make our lives a little easier, we use the following lemma.

\begin{lemma} \label{LemmaSufficient}
It suffices to check Thm.\ \ref{ThmFlat} and Thm.\ \ref{ThmFredholm} for one star-shaped neighborhood only.
\end{lemma}

\begin{proof}
Suppose that the theorems have been established for some open neighborhood $U$ star-shaped with respect to $X$ around $p$. We claim that in this case, they already hold for any other star-shaped set $U^\prime$. 

First observe that there is a solution at all on $U^\prime$: Choose a "sphere" with respect to $X$ in $U \cap U^\prime$, by which we mean an $n-1$-dimensional submanifold $S$ of $M$ such that for each point $q \in U^\prime$, there is exactly one $t_0$ such that $\Phi_{t_0}(q) \in S$ ($\Phi_t$ being the flow of $X$). If $u \in \Gamma^\infty(U, \V)$ is a solution of \eqref{InhomogeneousEquation} on $U$, then a solution on $U^\prime$ can be constructed with the method of characteristics by solving the ordinary differential equation
\begin{equation}
  \frac{\dd}{\dd t} u^\prime(t) + A(\Phi_t(q))u^\prime(t) = \lambda u^\prime(t) + v(\Phi_t(y)), ~~~~~~ u^\prime(t_0) = u(y)
\end{equation}
 along the outward integral curves of $X$ with initial values given by $u|_S$; this is then the only solution on $U^\prime$ that coincides with $u$ on $U$. 

In the same way, the restriction of any solution $u^\prime$ on $U^\prime$ to $U \cap U^\prime$ can  be uniquely extended to a solution on $U$.
\end{proof}

\begin{corollary}
We may assume $M \subseteq \R^n$ open and that $\V = M \times V$ for some fixed vector space $V$. Furthermore, we may assume that the connection on $\V$ is trivial, $\nabla = \DD$, so in the case of Thm.\ \ref{ThmFlat}, we are left with the differential equation
\begin{equation} \label{TheEquation3}
  (\DD_X + A)u = v.
\end{equation}
\end{corollary}

\begin{proof}
Choose a neighborhood $U$ of $p$ that is star-shaped with respect to $X$ and lies in the domain of a chart. By the above discussion, it suffices to verify Thm.\ \ref{ThmFlat} on this neighborhood; hence we may assume $M \subseteq \R^n$ open and $\V = M \times V$, as $\V$ must be trivial on $U$. By replacing
\begin{equation*}
   A^k_j(y) + X^i(y)\Gamma^k_{ij}(y) ~~~~~\rightsquigarrow~~~~~ A^k_j(y),
\end{equation*}
we need only to consider equations of the form \eqref{TheEquation3}.
\end{proof}

\medskip

By the above observations, we may assume that $M = \R^n$ and $\V = \R^n \times V$ for some fixed vector space $V$. Let $X \in \Cinf(U, \R^n)$ have a strictly positive source at $0$ and assume that $U$ is open and star-shaped around $0$ with respect to $X$. Let $A \in \Cinf(U, \End(V))$ and let $E(t, y)$ be the unique matrix solution of the problem
  \begin{equation*}
    \frac{\partial}{\partial t}E(t, y) = - A(\Phi_t(y))E(t, y), ~~~~~~~ E(0, y) = \id,
  \end{equation*}
$\Phi_t(y)$ denoting the flow of $X$. Then
  \begin{equation} \label{SolutionIdentities}
    E(s, \Phi_t(y)) = E(s+t, y)E(t, y)^{-1}
  \end{equation}  
which is easy to verify using uniqueness of solutions. Note also that $E(t, y)^{-1}$ solves the problem
\begin{equation*}
  \frac{\partial}{\partial t}[E(t, y)^{-1}] = E(t, y)^{-1} A(\Phi_t(y)), ~~~~~~~ E(0, y)^{-1} = \id,
\end{equation*}
so Thm.\ \ref{ThmODESolutionEstimate} gives similar estimates for $E^{-1}(t, y)$ as for $E(t, y)$.

\medskip

{\em The solution.} We claim that
  \begin{equation} \label{TheSolution}
    \fbox{$u(y) = \displaystyle\int_{-\infty}^0 E(s, y)^{-1} v(\Phi_s(y)) \dd s$}
  \end{equation}
 is a solution to the differential equation \eqref{TheEquationFlat}. Assuming for a moment that the integral converges absolutely for all $y \in U$, we calculate
 \begin{align*}
 u(\Phi_t(y)) &= \int_{-\infty}^0 E(s, \Phi_t(y))^{-1} v(\Phi_{s+t}(y)) \dd s
 \stackrel{\text{\eqref{SolutionIdentities}}}{=} \int_{-\infty}^0 \bigl(E(s+t, y)E(t, y)^{-1}\bigr)^{-1} v(\Phi_{s+t}(y)) \dd s \\
 &= \int_{-\infty}^0 E(t, y) E(s+t, y)^{-1} v(\Phi_{s+t}(y)) \dd s
 = \int_{-\infty}^t E(t, y) E(s, y)^{-1} v(\Phi_{s}(y)) \dd s
 \end{align*}
 for all $t$ and $y$ such that $\Phi_t(y) \in U$ so that
  \begin{align*}
  \DD_X u(y) &= \left.\frac{\dd}{\dd t} \right|_{t=0} u(\Phi_t(y))
  = \left.\frac{\dd}{\dd t} \right|_{t=0}\int_{-\infty}^t E(t, y) E(s, y)^{-1} v(\Phi_{s}(y)) \dd s \\
  &= v(\Phi_0(y)) - A(y) \int_{-\infty}^0 E(0, y) E(s, y)^{-1} v(\Phi_s(y)) \dd s
   = v(y) - A(y) u(y)
  \end{align*}
as desired. 

It remains to show that $u$ is well-defined, smooth and that it is the only flat solution. A key to this will be the following lemma.

\begin{lemma}[Characterization of flat Functions] \label{FlatSolutionsEstimate}
 Let $X \in \Cinf(U, \R^n)$ be a vector field with a strictly positive source at $0$ and assume that $U$ is star-shaped around $0$ with respect to $X$. Let 
 \begin{equation}
   \nu < \min \Re\, \mathrm{spec}\, \DD X|_0.
\end{equation}
 A function $u \in \Cinf(U, V)$ is in $\m^{N+1}$, the space of functions vanishing of order $N$ at $0$ (see Notation {\normalfont \ref{NotationPolynomialSpaces}}), if and only if for each compact neighborhood $K \subset U$ of $0$, there exists a constant $C>0$ such that
  \begin{equation} \label{EqVanishingIdealEstimate}
    |u(\Phi_t(y))| \leq Ce^{t N \nu}|y|^N
  \end{equation}
  for all $t\leq0$ and $y \in K$. Furthermore, $u$ is flat at $0$ if and only if for each $S \in \R$ and each compact neighborhood $K \subset U$, there exists a $C>0$ such that
  \begin{equation} \label{EqFlatsolutionsEstimate}
    |u(\Phi_t(y))| \leq C e^{St}
  \end{equation}
  for all $t\leq 0$ and $y \in K$.
\end{lemma}

\begin{proof}
Because $X$ has a strictly positive source at $0$, for each compact neighborhood $K \subset U$ of $0$, there exist constants $C, \delta, \mu >0$ such that
\begin{equation*}
  \delta e^{\mu t}|y| \leq |\Phi_t(y)| \leq C e^{\nu t}|y|
\end{equation*}
for all $y \in K$ and all $t\leq 0$ (see e.g.\ \citep[p.\ 105ff]{perko91}).
Because $u \in \m^{N+1}$, there exists a constant $C^\prime>0$ such that $|u(y)| \leq C^\prime |y|^N$. Now for $t\leq 0$,
  \begin{equation*}
    |u(\Phi_t(y))| \leq C^\prime | \Phi_t(y)|^N \leq C^\prime C^N e^{\nu N t} |y|^{ N}.
  \end{equation*}
Conversely, if $u \notin \m^{N+1}$, then $|u(y)| \geq C^\prime |y|^{N-1}$ for some $C^\prime>0$ and $y$ near zero. This gives a contradiction with \eqref{EqVanishingIdealEstimate} by fixing $t$ and letting $y$ tend to zero.

By invoking this result for all $N$, one obtains one direction of the result on flat functions. Now let \eqref{EqFlatsolutionsEstimate} hold and assume that $u$ is not flat at zero. Then there is some $N \in \N$ such that for any $C^\prime>0$, we have $|u(y)| \geq C^\prime |y|^N$ if only $y$ is small enough (depending on $C^\prime$). In particular, for any $C^\prime$ and $y \in K$, we have $|u(\Phi_t(y))| \geq C^\prime |\Phi_t(y)|^N$ whenever $t$ is small enough, depending on $y$ and $C^\prime$. Hence
\begin{equation*}
  |u(\Phi_t(y))| \geq C^\prime |\Phi_t(y)|^N \geq C^\prime \delta^N e^{\mu N t} |y|^N
\end{equation*}
for such $t$. This shows that if we set $S := \mu N$, \eqref{EqFlatsolutionsEstimate} does not hold for all $t\leq0$, regardless of the choice of $C$.
\end{proof}

Let us now finish the proof of Theorem \ref{ThmFlat}. 

\medskip

{\em 1. Smoothness}. For $N \in N$, define
\begin{equation}
  u_N(y) := \int_{-N}^0 E(t, y)^{-1} v(\Phi_t(y)) \dd t.
\end{equation}
$E$, $\Phi_t$ and $v$ are all smooth in $y$, so we have clearly $u_N \in C^\infty(U, V)$. We want to show that for each compact subset $K$ of $U$ and each $m \in \N_0$, $u_N$ is a Cauchy sequence in $C^m(K, V)$. To this end, we compute
\begin{equation}
  \DD^\alpha u_N(y) = \sum_{\beta \leq \alpha} \binom{\alpha}{\beta}\int_{-N}^0 \DD^{\alpha-\beta}[E(t, y)^{-1}] \DD^\beta [ v(\Phi_t(y))] \dd t.
\end{equation}
With Thm.\ \ref{ThmODESolutionEstimate} and Corollary \ref{FlowEstimatesII}, this can be estimated by
\begin{equation*}
  |\DD^\alpha u_N(y) - \DD^\alpha u_M(y)| \leq C \sum_{|\alpha| \leq m}\left|\int_{-N}^{-M} e^{\nu t}|\DD^\alpha v(\Phi_t(y))| \dd t\right|
\end{equation*}
for some $C>0$ and $\nu \in \R$. Now by assumption, $v$ (and therefore also its derivatives) are flat at zero, so that we can apply Lemma \ref{FlatSolutionsEstimate}. Doing so, we choose $S$ such that $S + \nu > 0$ and constants $C_\alpha>0$,
\begin{equation*}
  |\DD^\alpha v(\Phi_t(y))| \leq C_\alpha e^{St}
\end{equation*}
for all $t \leq 0$ and $y \in K$ to obtain
\begin{equation*}
  |\DD^\alpha u_N(y) - \DD^\alpha u_M(y)| \leq C^\prime | e^{-N(S+\nu)} - e^{-M(S+\nu)}|,
\end{equation*}
for some new constant $C^\prime$ and all $y \in K$, which shows that $u_N$ is indeed a Cauchy sequence with respect to every $C^m(K)$ norm. 

\medskip

{\em Step 2: $u$ is flat.} By the same considerations as above, we have
\begin{equation*}
  |u(\Phi_s(y))| \leq C \int_{-\infty}^s e^{\nu t} |v(\Phi_{t}(y))| \dd t
\end{equation*}
for some $C>0$ and $\nu \in \R$. Given $S>0$, we can choose $S^\prime$ such that $S^\prime + \nu > S$ and $C^\prime$ such that $|v(\Phi_t(y))| \leq C^\prime e^{S^\prime t}$ to obtain
\begin{equation*}
  |u(\Phi_t(y))| \leq \frac{CC^\prime}{S^\prime+\nu} e^{(S^\prime + \nu)t} \leq C^{\prime\prime} e^{St}
\end{equation*}
for all $t \leq 0$ and $y \in K$. Again by \ref{FlatSolutionsEstimate}, $u$ is flat at $0$.

\medskip
  
{\em Step 3: Uniqueness.} Let $u_1, u_2$ be two flat functions satisfying $(\DD_X + A)u_i = v$. Then $w:= u_1 - u_2$ is flat and  a solution to the homogeneous equation. This means that for every $y \in U$, the function $t \mapsto w(\Phi_t(y))$ solves the ordinary differential equation
\begin{equation*}
  \frac{\partial}{\partial t} w(\Phi_t(y)) = - A(\Phi_t(y))\cdot w(\Phi_t(y)), ~~~~~~~ w(\Phi_0(y)) = w(y),
\end{equation*}
hence $w(\Phi_t(y)) = E(t, y)w(y)$. As remarked above, we can apply Thm.\ \ref{ThmODESolutionEstimate} to the inverse $E(t, y)^{-1}$ and obtain that there exists $C, \nu \in \R$ such that 
\begin{equation*}
|E(t, y)^{-1}u| \leq C e^{\nu t}|u|
\end{equation*}
for all $u \in V$, $t \leq 0$ and $y$ in a compact neighborhood $K$ of $p$. Setting $u=E(t, y)w(y)$, we get
\begin{equation*}
   | w(\Phi_t(y)) | = \left| E(t, y) w(y) \right|  \geq \frac{1}{C}  e^{-\nu t} |w(y)|
\end{equation*}
for all $t<0$ and by Lemma \ref{FlatSolutionsEstimate}, $w$ is not flat unless $w(y)=0$ for all $y$.

\medskip

The proof of Thm.\ \ref{ThmFlat} is now complete. $\Box$


\section{Power Series and Correspondence} \label{section3}

Because of the discussion at the beginning of section \ref{section2}, it suffices to consider the equation 
\begin{equation} \label{TheEquation4}
  (\DD_X + A)u = \lambda u + v, ~~~~~~~ u \in \Cinf(U, V)
\end{equation}
i.e.\ we assume $U \subseteq \R^n$ and $\V = U \times V$ with some fixed vector space $V$. Again fix a vector field $X \in \Cinf(U, \R^n)$ with a strictly positive source at $0$ and assume that $U$ is star-shaped around $0$ with respect to $X$. 

\begin{notation} \label{NotationPolynomialSpaces}
For an open neighborhood $U$ of $0$, define
\begin{equation*}
  \m^{N} := \left\{ u \in \Cinf(U, V) \mid \DD^\alpha u(0) = 0 ~\text{for all}~ |\alpha| < N\right\},
\end{equation*}
the space of functions that vanish to order $N$ at $0$ (this is a closed subspace). Set furthermore
\begin{equation*}
  \mathcal{P}_N := \Cinf(U, V) / \m^{N+1} ~~~~ \text{and} ~~~~ \mathcal{H}_N := \m^N / \m^{N+1}
\end{equation*}
for the factor spaces isomorphic to spaces of polynomials and homogeneous polynomials, respectively. Generic elements of $\mathcal{P}_N$ and $\mathcal{H}_N$ will be denoted with fat letters (e.g.\ $\boldsymbol{u}$) and for a function $u \in \Cinf(U, V)$, we write $[u]_N$ for the corresponding elements in the factor spaces (no confusion should arise here). Furthermore, by $[u]_\infty$, we denote the Taylor series of $u$ at $0$.
\end{notation}

Equation \eqref{TheEquation4} descends to an equation
\begin{equation} \label{TheEquation5}
  ([\DD_X + A]_\infty - \lambda)\sum_{\alpha} u_\alpha y^\alpha = [v]_\infty,
\end{equation}
on the space $V[[y]]$ of Taylor series with values in $V$, where $[\DD_X + A]_\infty$ is defined by the formula
\begin{equation*}
  [\DD_X + A]_\infty[u]_\infty = [(\DD_X + A)u]_\infty.
\end{equation*}
This is well-defined as $\DD_X + A$ maps flat functions to flat functions. Notice furthermore that the operator $[\DD_X + A]_N$ is also well-defined on both $\mathcal{P}_N$ and $\mathcal{H}_N$ by a similar formula because $\DD_X + A$ maps $\m^N$ to $\m^N$ for all $N \in N$.

The purpose of this section is to prove the following theorem, which implies Thm.\ \ref{ThmEigenvalues} and is an important step for the proof of Thm.\ \ref{ThmFredholm}.

\begin{theorem}[Correspondence] \label{ThmCorrespondence}
Let $\lambda \in \C$ and $N \in \N$. In the case that $\lambda$ is of the form
\begin{equation*}
  \lambda = \alpha_1 \mu_1 + \dots + \alpha_n \mu_n + \rho,
\end{equation*}
$\mu_1, \dots, \mu_n$ being the (generalized) eigenvalues of $\DD X|_0$ and $\rho$ being an eigenvalue of $A(0)$, suppose that $N$ is big enough such that for all representations of $\lambda$ in the above form, we have $|\alpha| \leq N$. Then there is a one-to-one correspondence between solutions of \eqref{TheEquation4} and solutions of the projected equation
\begin{equation} \label{TheEquationPr}
  [\DD_X + A]_N \boldsymbol{u} = \lambda \boldsymbol{u} + [v]_N
\end{equation}
with $\boldsymbol{u} \in \mathcal{P}_N$. More precisely, the map $u \mapsto [u]_N$ is an isomorphism of affine spaces between the solution spaces of \eqref{TheEquation4} and \eqref{TheEquationPr}.
\end{theorem}

\begin{remark} \label{RemarkSolvability}
Notice that \eqref{TheEquationPr} is solvable if and only if $[v]_N \in \mathrm{im} ([\DD_X + A]_N - \lambda)$.
\end{remark}

The strategy of the proof of Thm.\ \ref{ThmCorrespondence} is the following: First we construct formal power series solutions of \eqref{TheEquation4}. Then we use Borel's theorem to obtain a function that solves \eqref{TheEquation4} up to a flat function. Finally we use Thm.\ \eqref{ThmFlat} to correct this function, giving an actual solution. Let us first recall Borel's theorem.

\begin{theorem}[Borel] \label{BorelsTheorem}
Let $\sum_\alpha u_\alpha y^\alpha \in V[[y]]$ be a formal power series with coefficients $u_\alpha \in V$, where $V$ is some Banach space. Then for each open neighborhood $U$ of $0$, there exists a function $u \in \Cinf(U, V)$ whose Taylor series at $0$ coincides with the given formal power series.
\end{theorem}

\begin{proof}[sketch]
Take a smooth cutoff function $\chi:[0,\infty) \longrightarrow [0,1]$ with $\chi \equiv 1$ on $[0,1]$ and $\chi \equiv 0$ on $[2, \infty)$. Now define
\begin{equation*}
  u(y) := \sum_{j=0}^\infty \chi(\gamma_{|\alpha|} \cdot |y|) u_\alpha y^\alpha
\end{equation*}
where $\gamma_j \longrightarrow \infty$. $u$ is well-defined, because the sum is in fact only a finite sum, and if one chooses the sequence $(\gamma_j)$ to increase fast enough, $u$ can be made smooth. For a proof that this is possible, see e.g.\ \citep{narashimhan1985}, Thm.\ 1.5.4 and the lemma before.
\end{proof}

\begin{lemma} \label{LemmaOnPowerSeries}
Suppose that the formal power series $\sum_\alpha u_\alpha y^\alpha$ solves \eqref{TheEquation5}. Then for each neighborhood $U$ of $0$ that is star-shaped with respect to $X$, there exists a unique function $u \in \Cinf(U, V)$ with $[u]_\infty = \sum_\alpha u_\alpha y^\alpha$ solving \eqref{TheEquation4}.
\end{lemma}

\begin{proof}
By Thm.\ \ref{BorelsTheorem}, there exists a function $\tilde{u} \in \Cinf(U, V)$ with Taylor series $\sum_\alpha u_\alpha y^\alpha$. Then
\begin{equation*}
  r := (\DD_X + A)\tilde{u} - \lambda \tilde{u} - v
\end{equation*}
is flat at $0$. By Thm.\ \ref{ThmFlat}, there exists a a unique flat function $w \in \Cinf(U, V)$ such that $(\DD_X + A- \lambda)w = r$. Set $u = \tilde{u} - w$, then
\begin{equation*}
  (\DD_X + A)u = (\DD_X + A)\tilde{u} - \lambda w - r = r + \lambda \tilde{u} + v - \lambda w - r = \lambda u +v.
\end{equation*}
To see that this solution is unique, assume we had two solutions $u_1, u_2$ with the same Taylor series $\sum_\alpha u_\alpha y^\alpha$ of \eqref{TheEquation4}. Then $w := u_1 - u_2$ is flat and solves the homogeneous equation
\begin{equation*}
  (\DD_X + A - \lambda)w = 0.
\end{equation*}
From uniqueness of solutions in Thm.\ \ref{ThmFlat}, we conclude $w \equiv 0$.
\end{proof}

\begin{lemma} \label{LemmaEigenvalues}
The generalized eigenvalues (i.e.\ zeros of the characteristic polynomial) of the operator $[\DD_X + A]_N$ on the space $\mathcal{P}_N$ are exactly the numbers
\begin{equation*}
  \lambda = \alpha_1 \mu_1 + \dots + \alpha_n \mu_n + \rho,
\end{equation*}
where $\mu_j$ are the eigenvalues of the linearization $\DD X|_0$, $\rho$ is an eigenvalue of $A(0)$ and $\alpha$ is a multi-index with $|\alpha| \leq N$.
\end{lemma}

Before we give the proof, we calculate an example to which we will come back later.

\begin{example} \label{JordanNormalFormOnP2}
Let us assume we are in $\R^2$ and deal with scalar functions. Let $\phi(y) = \frac{1}{2} y_1^2+ y_1^2y_2 + y_2^2$ and let $X = \grad \phi$. A basis of the six-dimensional space $\mathcal{P}_2$ is 
\begin{equation} \label{BasisOfP2}
  [1], [y_1], [y_2], [y_1^2], [y_1y_2], [y_2^2].
\end{equation}
A little calculation shows that the matrix representation of $[\DD_X]_2$  with respect to this basis is
\begin{equation} \label{MatrixRepresentationOnP2}
[\DD_X]_2 \hat{=}
\begin{pmatrix}
0 &   &   &   &   & \\
0 & 1 &   &   &   & \\
0 & 0 & 2 &   &   & \\
0 & 0 & 1 & 2 &   & \\
0 & \circled{2} & 0 & 0 & 3 & \\
0 & 0 & 0 & 0 & 0 & 4
\end{pmatrix}.
\end{equation}
The general observation will be that the matrix representing $[\DD_X + A]_N$ with respect to a similar basis as \eqref{BasisOfP2} is a lower triangular matrix.

If one replaces the basis vector $[y_1]$ by $[y_1 - y_1y_2]$ in the basis \eqref{BasisOfP2}, one obtains the same matrix, except that the  number $2$ encircled in \eqref{MatrixRepresentationOnP2} disappears, so that $[\DD_X]_2$ is in Jordan normal form with respect to this new basis. In particular, this shows that $[\DD_X]_2$ is not diagonalizable, as the algebraic multiplicity of the eigenvalue $2$ is two, but its geometric multiplicity is just one. In fact, $[\DD_X]_N$ will not be diagonalizable on $\mathcal{P}_N$ for any $N$, which illustrates where the failure of $\DD_X$ to have the "expected" multiplicity comes from.
\end{example}

\begin{proof}[of Lemma {\normalfont \ref{LemmaEigenvalues}}]
Let $\DD X|_0 = (a_i^j)$ and set
\begin{equation*}
D_0 = a^j_i y^i \DD_j.
\end{equation*}
We then have
\begin{equation*}
  [\DD_X]_N = D_0 + \boldsymbol{D} ~~~~~~ \text{and} ~~~~~~ [A]_N = A(0) + \boldsymbol{A},
\end{equation*}
where both $\boldsymbol{D}$ and $\boldsymbol{A}$ increase the order of polynomials at least by $1$, hence are nilpotent on the space $\mathcal{P}_N$. 

In the following, we assume the Jordan normal form of an endomorphism to be a {\em lower} triangular matrix.

Choose (Euclidean) coordinates $y$ such that $\DD X|_0$ is in real Jordan normal form with respect to the basis $\partial_{y^j}$ of $T_0 U$ and enumerate the eigenvalues of $\DD X|_0$ such that $\mu_1, \dots, \mu_{2r}$ have non-vanishing imaginary part with $\mu_{2j-1} = \overline{\mu_{2j}}$ for $j = 1, \dots, r$ and $\mu_{2r+1}, \dots, \mu_n$ are real. Define the "virtual" coordinate functions
\begin{align*}
  z_{2j-1} &:= y_{2j-1} + i y_{2j}, ~~~ \text{and} ~~~ z_{2j} := y_{2j-1} - i y_{2j}  ~~~~~~~& &\text{for}~~~~1 \leq j \leq r \\
  z_j &:= y_j& &\text{for}~\,~2r < j \leq n.
\end{align*}
Then $\DD X|_0$ is in complex Jordan normal form with respect to the basis $\partial_{z_j}$, $j=1, \dots, n$. 

Let $b_1, \dots, b_m$ be a generalized eigenbasis of $A(0)$ to the eigenvalues $\rho_1, \dots \rho_m$, i.e.\ a basis such that $A(0)$ is in Jordan normal form with respect to this basis. Then the elements
\begin{equation} \label{BasisOfPN}
  [z^\alpha \cdot b_j]_N, ~~~~~ j = 1, \dots, m, ~~~~~ |\alpha| \leq N 
\end{equation}
form a basis of $\mathcal{P}_N$ that can be ordered in such a way (non-decreasingly in $|\alpha|$) that the the corresponding matrix representing the operator $D_0$ is in Jordan normal form, with eigenvalues 
\begin{equation*}
  \alpha_1 \mu_1 + \dots + \alpha_n \mu_n, ~~~~~|\alpha| \leq N.
\end{equation*}
The operator $\DD_0 + A(0)$ is in Jordan normal form as well with respect to this basis, and it has the eigenvalues
\begin{equation*}
  \alpha_1 \mu_1 + \dots + \alpha_n \mu_n + \rho_j, ~~~~~|\alpha| \leq N, ~~j = 1, \dots m.
\end{equation*}
Finally, 
\begin{equation*}
  [\DD_X + A]_N = D_0 + A(0) + \boldsymbol{D} + \boldsymbol{A}
\end{equation*}
has the same eigenvalues because $\boldsymbol{D}$ and $\boldsymbol{A}$ have only entries below the diagonal with respect to the basis above.
\end{proof}

\begin{proof}[of Thm.\ \ref{ThmCorrespondence}]It is clear that whenever $u \in \Cinf(U, V)$ solves \eqref{TheEquation4}, then $[u]_N \in \mathcal{P}_N$ solves \eqref{TheEquationPr}. We are left to show the converse. 

Suppose we are given a solution $\boldsymbol{u} \in \mathcal{P}_N$ of \eqref{TheEquationPr}. Choose a representative $u_N \in \Cinf(U, V)$ of $\boldsymbol{u}$, such that $[u_N]_N = \boldsymbol{u}$. Then
\begin{equation*}
  (\DD_X + A - \lambda)u_N - v =: r_{N+1} \in \m^{N+1}.
\end{equation*}
Now we construct a solution recursively. Suppose we have found $u_N, u_{N+1}, \dots u_{N+k-1} \in \Cinf(U, V)$ such that
\begin{equation*}
  (\DD_X + A - \lambda)(u_N + \dots + u_{N+k-1}) - v =: r_{N+k} \in \m^{N+k}.
\end{equation*}
Then we want to solve the equation
\begin{equation}
  ([\DD_X + A]_{N+k}- \lambda)\boldsymbol{u}_{N+k} = - [r_{N+k}]_{N+k},
\end{equation}
on the space $\mathcal{H}_{N+k}$. On this space, $[\DD_X + A]_{N+k}$ is in Jordan normal form with respect to the basis
\begin{equation*}
  [z^\alpha \cdot b_j] ~~~~~~ |\alpha|=N+k, ~~ j = 1, \dots, m,
\end{equation*}
as discussed the proof of \ref{LemmaEigenvalues}. The eigenvalues are the numbers
\begin{equation*}
  \alpha_1 \mu_1 + \dots + \alpha_n \mu_n + \rho_j, ~~~~~~ |\alpha|=N+k, ~~ j = 1, \dots, m.
\end{equation*}
This shows that that by the assumption on $N$, $\lambda$ is not an eigenvalue and $[\DD_X + A]_{N+k}- \lambda$ is invertible on $\mathcal{H}_{N+k}$, allowing us to set
\begin{equation*}
  \boldsymbol{u}_{N+k} := - ([\DD_X + A]_{N+k}- \lambda)^{-1}[r_{N+k}]_{N+k} \in \mathcal{H}_{N+k}.
\end{equation*}
Now choose a representative $u_{N+k}$ of $\boldsymbol{u}_{N+k}$ in $\Cinf(U, V)$; by construction,
\begin{equation*}
  (\DD_X + A - \lambda)(u_N + \dots + u_{N+k}) - v =: r_{N+k+1} \in \m^{N+k+1}.
\end{equation*}
Denote by $[u_{N+k}]_\infty$ the Taylor series of the respective terms. Since for each $k$, $[u_{N+k}]_\infty$ does not contain terms of order lower than $N+k$, the sum $\sum_{k=0}^\infty [u_{N+k}]_\infty =: \sum_{\alpha} u_\alpha y^\alpha$ defines an element of $V[[x]]$ solving \eqref{TheEquation5}. The Taylor series does not depend on the choices of the individual $u_{N+k}$, because for two such choices $u_N, u_{N+1}, \dots$, and $\tilde{u}_{N}, \tilde{u}_{N+1}, \dots$, we have by construction
\begin{equation*}
  \sum\nolimits_{k=0}^M u_{N+k} - \sum\nolimits_{k=0}^M \tilde{u}_{N+k} \in \m^{N+M+1}.
\end{equation*}
By Lemma \ref{LemmaOnPowerSeries}, the power series constructed this way gives rise to a unique solution $u \in \Cinf(U, V)$ of \ref{TheEquation4}.
\end{proof}

We can now prove Thm.\ \ref{ThmEigenvalues}.

\begin{proof}[of Thm.\ \ref{ThmEigenvalues}]
Suppose that $\lambda$ is not of the form \eqref{TheEigenvalue}. Then by Thm.\ \ref{ThmCorrespondence}, every solution of $(\DD_X + A)u = \lambda u$ solves also
\begin{equation*}
  [\DD_X + A]_1[u]_1 = \lambda[u]_1
\end{equation*}
on $\mathcal{P}_1$. By Lemma \ref{LemmaEigenvalues}, $[u]_1=0$ and then by Thm.\ \ref{ThmCorrespondence}, $u=0$, so $\lambda$ is not an eigenvalue of $\DD_X + A$.

If $\lambda$ is of the form \eqref{TheEigenvalue} with "combinatorial multiplicity" $m(\lambda)$, then for $N$ big enough as in Thm.\ \ref{ThmCorrespondence}, it is an eigenvalue on all spaces $\mathcal{P}_N$ of algebraic multiplicity $m(\lambda)$, so there exists a geometrical eigenspace to $\lambda$ of dimension at least $1$ and at most $m(\lambda)$ and by Thm.\ \ref{ThmCorrespondence}, the same is true on $\Gamma^\infty(U, \V)$.

Conversely, if $\lambda$ is an eigenvalue of $\DD_X + A$, then it must be of the form \eqref{TheEigenvalue}. Again by Thm.\ \ref{ThmCorrespondence}, the dimension of $\ker (\DD_X + A - \lambda)$ cannot be greater than $m(\lambda)$.
\end{proof}


\section{The dual Operator} \label{section3_5}

We now consider the dual equation \eqref{DualEquation}. 

\begin{remark} \label{RemarkDualOperator}
The dual operator is given by
\begin{equation*}
 (\nabla_X + A)^\prime = -\nabla_X^* + A^\prime - \div(X),
\end{equation*}
where $\nabla^*$ is the connection on $\V^*$ induced by $\nabla$. This formula can be taken literally when acting on {\em smooth} distributions $\Gamma_c^\infty(U, \V^*) \subset \Eh^\prime(U, \V^*)$.
\end{remark}

\begin{lemma}
Under the assumptions of Thm.\ \ref{ThmFlat}, let $T \in \Eh^\prime(U, \V^*)$ be a solution of the dual equation
\begin{equation} \label{DualEquation2}
  (\nabla_X + A)^\prime T = \lambda T.
\end{equation}
Then $\mathrm{supp}\, T \subseteq \{p\}$.
\end{lemma}

\begin{proof}
$T$ solves \eqref{DualEquation2} if and only if for all $u \in \Gamma^\infty(U, \V)$, we have
\begin{equation*}
  T\bigl( (\nabla_X + A - \lambda) u\bigr) = 0.
\end{equation*}
Let $v$ have support in $U \setminus \{p\}$. Then $v$ is flat at $0$, hence by Thm. \ref{ThmFlat}, there exists a unique flat function $u$ such that $(\nabla_X + A - \lambda) u = v$. Hence $T(v) = 0$ for all functions with compact support in $U \setminus \{p\}$, which is the statement.
\end{proof}

The standard characterization result of distributions with discrete supports gives the following corollary.

\begin{corollary} \label{CorollaryDelta}
Each solution $T$ of \eqref{DualEquation2} is a finite linear combination of derivatives of delta distributions $\delta_\xi$, where
\begin{equation*}
  \delta_\xi(u) = \xi(u(p)), ~~~~~~~ u \in \Gamma^\infty(U, \V).
\end{equation*}
for $\xi \in \V_p^*$.
\end{corollary}

Again we may assume that $U \subseteq \R^n$ open, that the vector field $X \in \Cinf(U, \R^n)$ has a strictly positive source at $0$ and that we consider functions with values in some fixed vector space $V$. Given a solution $T$ of $(\DD_X + A)^\prime T = \lambda T$, it has finite order $\leq N$ as a distribution. By the above corollary, we have $T(\m^{N+1})=0$ for this $N$, so $T$ descends to an element $[T]_N$ of $\mathcal{P}_N^\prime$. 

\begin{theorem}[Dual Correspondence] \label{ThmCorrespondence2}
Let $\lambda \in \C$ and $N \in \N$. In the case that $\lambda$ is of the form
\begin{equation*}
  \lambda = \alpha_1 \mu_1 + \dots + \alpha_n \mu_n + \rho,
\end{equation*}
$\mu_1, \dots, \mu_n$ being the eigenvalues of $\DD X|_0$ and $\rho$ being an eigenvalue of $A(0)$, suppose that $N$ is big enough such that for all representations of $\lambda$ in the above form with a multi-index $\alpha$ and an eigenvalue $\rho$ of $A(0)$, we have $|\alpha| \leq N$. 

Then there is a one-to-one correspondence between solutions of the dual equation \eqref{DualEquation2} and the projected equation
\begin{equation} \label{DualEquationPr}
  [\DD_X + A]_N^\prime \boldsymbol{T} = \boldsymbol{T} \circ [\DD_X + A]_N = \lambda \boldsymbol{T}
\end{equation}
on $\mathcal{P}_N^\prime$. More precisely, the map $T \mapsto [T]_N$ is an isomorphism between the corresponding eigenspaces.
\end{theorem}

\begin{proof}
Suppose that $T$ solves \eqref{DualEquation2}. By Cor.\ \ref{CorollaryDelta}, it is a finite linear combination of derivatives of Delta distributions, so it factors over the space $\mathcal{P}_N$ for $N$ big enough, i.e.\ for all $u \in C^\infty(U, V)$, we have $T(u) = \boldsymbol{T}([u]_N)$ for some $\boldsymbol{T} \in \mathcal{P}_N^\prime$.

Let $N^\prime$ be the lowest number that fulfills the requirement of the theorem. We first show that the distribution order of $T$ cannot be greater than $N^\prime$. As argued before, $T$ descends to an element $[T]_N \in \mathcal{P}^\prime_N$ for some $N \geq N^\prime$. With respect to the basis \eqref{BasisOfPN} of $\mathcal{P}_N$, $[T]_N$ is represented by a row vector and is a left eigenvector to $[\DD_X + A]_N$. As argued in the proof of lemma\ \ref{LemmaEigenvalues}, $[\DD_X + A]_N$ has the form
\begin{equation*}
  [\DD_X + A]_N ~\widehat{=}~ \begin{pmatrix} L & 0 \\ * & M \end{pmatrix},
\end{equation*}
with respect to this basis, where $L$ is an $N^\prime \times N^\prime$ matrix and $M$ does not have the eigenvalue $\lambda$. If now
\begin{equation*}
  [T]_N ~\widehat{=}~ ( v^t, w^t ), ~~~~~~ v \in \C^{N^\prime}, w \in \C^{N-N^\prime},
\end{equation*}
with respect to the same basis, then the fact that $T$ solves \eqref{DualEquation2} implies in particular $w^t (M-\lambda) = 0$ as $[T]_N$ is a left eigenvector. But this means $w=0$, as $M$ does not have the eigenvalue $\lambda$. Therefore, $T$ is actually of order less or equal to $N^\prime$ as a distribution.

This shows that every solution $T$ of \eqref{DualEquation2} descends to a unique solution $[T]_N$ of \eqref{DualEquationPr} in $\mathcal{P}_N$ for every $N \geq N^\prime$. 

Conversely, let $\boldsymbol{T}$ be a solution of \ref{DualEquationPr}. Then $\boldsymbol{T}$ defines a solution $T \in \Eh^\prime(U, V^*)$ of \eqref{DualEquation2} by setting $T(u) = \boldsymbol{T}([u]_N)$ for $N$ big enough. By the discussion before, $T$ is uniquely determined by this property.
\end{proof}

\begin{remark}
The solutions of the dual equation \eqref{DualEquation2} can be completely calculated, as one only needs to calculate finitely many terms, while the Taylor expansion of solutions $u$ of $(\DD_X + A)u = \lambda u$ has infinitely many terms in general.
\end{remark}

The proof of the theorem about the Fredholm alternative \ref{ThmFredholm} is now a consequence of the two correspondence theorems \ref{ThmCorrespondence} and \ref{ThmCorrespondence2}.

\begin{proof}[of Thm.\ \ref{ThmFredholm}]
Suppose first that $\lambda$ is not an eigenvalue (i.e.\ one of the numbers from Thm.\ \ref{ThmEigenvalues}). Then by Thm. \ref{ThmCorrespondence}, there is a one-to-one correspondence between solutions of \eqref{TheEquation4} and solutions of
\begin{equation*}
  ([\DD_X + A]_1 - \lambda) \boldsymbol{u} = [v]_1.
\end{equation*}
But by Lemma \ref{LemmaEigenvalues}, $\lambda$ is not an eigenvalue of $[\DD_X + A]_1$, so this equation has a unique solution.

Let now $\lambda$ be an eigenvalue. Then by Thm.\ \ref{ThmCorrespondence}, the dimension $k$ of the kernel of $\DD_X + A - \lambda$ coincides with the dimension of the kernel of $[\DD_X + A - \lambda]_N$ on $\mathcal{P}_N$, whenever $N$ is large enough. On the other hand, for every endomorphism of a finite-dimensional vector space, the number $k$ of right eigenvectors and left eigenvectors to a given eigenvalue coincide. Therefore, there are also $k$ solutions $\boldsymbol{T} \in \mathcal{P}_N^\prime$ of \eqref{DualEquationPr}, which by Thm.\ \ref{ThmCorrespondence2} give exactly the solutions of the dual equation \eqref{DualEquation2}.

The Fredholm alternative of linear algebra states that \eqref{TheEquationPr} is solvable if and only if $\boldsymbol{T}([v]_N) = 0$ for each $\boldsymbol{T} \in \ker ([\DD_X + A]_N^\prime- \lambda)$. But this means exactly that
\begin{equation*}
  v \in \ker ((\DD_X + A)^\prime - \lambda)_\perp.
\end{equation*}
\end{proof}

\begin{corollary} \label{CorollaryClosed}
For any open neighborhood $U$ that is starshaped around the strictly positive source $p$ of $X$, the operator $\nabla_X + A  - \lambda$ has closed range for every $\lambda \in \C$ and is a Fredholm operator of index $0$.
\end{corollary}

\begin{proof}
We have
\begin{equation*}
  \mathrm{im}(\nabla_X +A -\lambda) = \ker ( \nabla_X + A - \lambda)^\prime_\perp.
\end{equation*}
Because $\ker ( \nabla_X + A - \lambda)^\prime$ is finite-dimensional, the image of $\nabla_X + A -\lambda$ is the intersection of the kernels of finitely many distributions, hence it is closed.
\end{proof}

\begin{corollary} \label{ThmFredholmDual}
Under the assumptions of Thm.\ \ref{ThmFredholm}, if $\lambda \in \C$ admits a $k$-dimensional space of eigenfunctions, then the inhomogeneous equation
\begin{equation*}
  (\nabla_X + A)^\prime T = \lambda T + S
\end{equation*}
has a solution if and only if $S \in \ker(\nabla_X + A - \lambda)^\perp$, and in this case, the space of solutions is an affine space of dimension $k$.
\end{corollary}

\begin{proof}
  Set $L_\lambda := \nabla_X + A - \lambda$. By Cor.\ \ref{CorollaryClosed} before, $\im(L_\lambda)$ is closed for every $\lambda \in \C$, hence $\im(L_\lambda)$ is a Fréchet subspace. By Thm.\ \ref{ThmFredholm}, $\ker(L_\lambda)$ is finite-dimensional, hence closed and $W := \Gamma^\infty(U, \V) / \ker(L_\lambda)$ is Fréchet as well. For $S \in \ker(L_\lambda)^\perp$, the functional $\overline{S} \in W^\prime$ is well-defined by
  \begin{equation*}
    \overline{S}(u + \ker(L_\lambda)) := S(u)
  \end{equation*}
  because when $v \in u + \ker(L_\lambda)$, then $u-v \in \ker(L_\lambda)$, hence $S(u) = S(v)$.
  Now $\overline{L}_\lambda: W \longrightarrow \im(L_\lambda)$ defined by
  \begin{equation*}
    \overline{L}_\lambda(u + \ker(L_\lambda)) = L_\lambda(u)
  \end{equation*}
  is bijective. On the other hand, $L_\lambda$ is a differential operator and therefore continuous on $\Gamma^\infty(U, \V)$ and so is $\overline{L}_\lambda$. By the bounded inverse theorem, $\overline{L}_\lambda^{-1}$ is continuous as well. Therefore $T$ defined by
  \begin{equation*}
    \overline{T}(u) := \overline{S}\bigl(\overline{L}_\lambda^{-1}(u)\bigr),
  \end{equation*}
  is a well-defined continuous functional on $\im(L_\lambda)$. Because $\im(L_\lambda)$ has codimension $k$, the affine space of functionals $T \in \Eh^\prime(U, \V^*)$ such that $T|_{\im(L_\lambda)} \equiv \overline{T}$ has dimension $k$ as well.
\end{proof}


\section{Smoothness of the Solution Map} \label{sectionsmooth}

In this section is to show that the solution $u \in \Gamma^\infty(U, \V)$ to the equation
\begin{equation*}
 (\nabla_X + A)u = v
\end{equation*}
depends smoothly on the coefficients, i.e.\ the vector field $X$, the endomorphism field $A$, the connection $\nabla$ and the right side $v$, meaning that the map sending the data $(\nabla, X, A, v)$ to the solution $u$ is a smooth map between the involved Fréchet spaces.

Again, it it suffices to consider the problem for $U \subseteq \R^n$, with a trivial bundle and trivial connection; the Christoffel symbols of the connection can be absorbed into the endomorphism $A$ as explained in section \ref{section2}. We consider the operator
\begin{align*}
\mathbf{T}: \X_U \times \Cinf&(U, \End(V)) \times \Cinf(U, V) \longrightarrow \Cinf(U, V)\\
&(X, A, u) \longmapsto (\DD_X + A)u,
\end{align*}
where $\X_U$ was defined in \ref{NotationXU}.
It is clear that in the interior of its domain, $\mathbf{T}$ is smooth with respect to the Fréchet topology. Denote by $\mathscr{U} \subseteq \mathrm{int}(\X_U) \times \Gamma^\infty(U, \End(\V))$ the set of data $(X, A)$ such that the operator $\nabla_X + A$ has a trivial kernel. By Thm.\ \ref{ThmEigenvalues}, this is equivalent to a condition on the eigenvalues of $A(p)$ and $\DD X|_p$ (where $p$ is the strictly positive source of $X$), namely the they cannot be combined in the form \eqref{TheEigenvalue} to give zero. This shows that the set $\mathscr{U}$ is open, because eigenvalues depend continuously on the operator (see \cite[Thm.\ II.5.14]{kato95}).

For $(X, A) \in \mathscr{U}$, we know by Thm.\ \ref{ThmFredholm} that for each $v \in \Cinf(U, V)$, there is a unique $u \in \Cinf(U, V)$ such that
\begin{equation} \label{TheEquation17}
  \mathbf{T}(X, A, u) = v.
\end{equation} 
Therefore we have a well-defined solution operator 
\begin{equation*}
  \mathbf{S}: \mathscr{U} \times \Cinf(U, V) \longrightarrow \Cinf(U, V)
\end{equation*}
that maps the data $(X, A, v)$ to the solution $u$ of \eqref{TheEquation17}.

\begin{theorem} \label{SmoothDependence}
The solution operator $\mathbf{S}$ is smooth with respect to the Fréchet topology.
\end{theorem}

We first prove a slightly weaker theorem and show subsequently how to adapt the proof in the more general situation. Let $\mathscr{U}^\prime \subset \mathscr{U}$ be the subset of data such that if $p$ is the positive definite source of $X$, we have
\begin{equation} \label{Apositive}
  0 < \min \Re\, \mathrm{spec} \,A(p).
\end{equation}

\begin{lemma} \label{LemmaSmoothDependence}
The solution operator $\mathbf{S}$ is smooth when restricted to the subset $\mathscr{U}^\prime$.
\end{lemma}

\begin{proof}
Let again $E_{X, A}$ be the unique matrix solution of the problem
  \begin{equation} \label{InitProblem77}
    \frac{\partial}{\partial t}E_{X, A}(t, y) = - A(\Phi_t(y))E_{X, A}(t, y), ~~~~~~~ E(0, y) = \id
  \end{equation}
We claim that $\mathbf{S}$ is given by
  \begin{equation} \label{TheSolution2}
    \mathbf{S}(X, A)v = \int_{-\infty}^0 E_{X, A}(t, y)^{-1} v(\Phi_t^X(y)) \dd t,
  \end{equation}
which can be verified formally as in Section \ref{section2}. However, we need to check that this actually defines a function in $\Cinf(U, V)$ when $v$ is not flat. Similar as in Section \ref{section2}, we define for $N \in \N$
  \begin{equation*}
    \mathbf{S}_N(X, A)v = \int_{-N}^0 E_{X, A}(t, y)^{-1} v(\Phi_t^X(y)) \dd t,
  \end{equation*}
which is smooth as both $E$ and $\Phi$ depend smoothly on all data. Now we use that for $(X, A) \in \mathscr{U}^\prime$ with strictly positive source $p \in U$, the spectra of both $A(p)$ and $\DD X|_{p}$ are bounded from below by a constant $\nu >0$ (and the same is true for $(X^\prime, A^\prime)$ in some neighborhood of $(X, A)$ in $\mathscr{U}^\prime$). By \ref{ThmODESolutionEstimate} and \ref{FlowEstimatesII}, we get that for each compact subset $K$ of $U$, there exists a constant $C>0$ such that
\begin{equation} \label{EstimatesOnEandPhi}
  \|E_{X^\prime, A^\prime}(t)\|_{C^m(K)} \leq C e^{\nu t} ~~~~ \text{and}~~~~ \|\DD \Phi^{X^\prime}_t\|_{C^m(K)} \leq C e^{\nu t}
\end{equation}
for all $(X^\prime, A^\prime)$ in some $C^m(K)$-ball around $(X, A)$. 

The derivatives of the integrand are given by
\begin{equation*}
  \sum_{\beta \leq \alpha} \binom{\alpha}{\beta} \DD^{\alpha-\beta}[E_{X, A}(t, y)^{-1}] \DD^\beta [ v(\Phi_t^X(y))].
\end{equation*}
Using the estimates \eqref{EstimatesOnEandPhi} above, we get that
\begin{equation*}
  \| \mathbf{S}_N(X^\prime, A^\prime)v - \mathbf{S}_M(X^\prime, A^\prime)v\|_{C^m(K)} \leq C^\prime \|v\|_{C^m(K)} |e^{-\nu N}- e^{-\nu M}|.
\end{equation*}
for all $X^\prime$, $A^\prime$ in some $C^m$-neighborhood of $(X, A)$ and all $v \in \Cinf(U, V)$ (where $C^\prime$ is some new constant independent of $N$ and $M$). This shows that $\mathbf{S}_N$ converges locally uniformly to $\mathbf{S}$, hence $\mathbf{S}$ is continuous.

The smoothness follows now from the following general result: Let $\mathbf{T}: (\mathscr{U} \subseteq \mathcal{F}) \times \mathcal{G} \longrightarrow \mathcal{H}$ be a smooth family of continuous linear maps between Fréchet spaces (i.e.\ $T(f): \mathcal{G} \longrightarrow \mathcal{H}$ is continuous and linear for each $f \in \mathscr{U}$) with continuous inverse $\mathbf{S}: \mathscr{U} \times \mathcal{H} \longrightarrow \mathcal{G}$, then $\mathbf{S}$ is smooth (see \cite[Thm.\ I.5.3.1]{hamilton}).
\end{proof}

\begin{corollary}
For $(X, A) \in \mathscr{U}^\prime$ and $v \in \Gamma^\infty(U, \V)$, the unique solution $u$ to the equation $(\DD_X + A)u = v$ is given by the integral formula
\begin{equation} \label{SolutionFormula}
  u(y) =  \int_{-\infty}^0 E_{X, A}(t, y)^{-1} v(\Phi_t^X(y)) \dd t,
\end{equation}
where $E_{X, A}$ is the solution to \eqref{InitProblem77} and $\Phi_t^X$ is the flow of $X$.
\end{corollary}

\begin{remark}
From the formula \eqref{TheSolution2} and the estimates from \ref{ThmODESolutionEstimate} and \ref{FlowEstimatesII}, it follows that for each compact subset $K$ of $U^\prime$ and each $(X, A) \in \mathscr{U}^\prime$, there exists a neighborhood $\mathscr{V} \subseteq \mathscr{U}^\prime \times C^\infty(U, V)$ such that
\begin{equation*}
 \| \mathbf{S}(X^\prime, A^\prime)v^\prime \|_{C^m(K)} < C \bigl(1 + \|v^\prime\|_{C^m(K)} + \|X^\prime\|_{C^m(K)} + \|A^\prime\|_{C^m(K)}\bigr)
\end{equation*}
for all $(X^\prime, A^\prime, v^\prime) \in \mathscr{V}$,
showing that $\mathbf{S}$ is in fact a {\em tame} map when interpreted as acting between the tame Fréchet spaces $C^\infty(K, \R^n)$, $C^\infty(K, \End(V))$, $\Cinf(K, V)$. This is a result analogous result to the case of the solution map to a linear elliptic operator (compare \citep[Thm.\ II.3.3.1]{hamilton}). 
\end{remark}

\begin{proof}[of Thm.\ \ref{SmoothDependence}]
Now assume that for $(X, A) \in \mathscr{U}$, we only have the weaker estimate
\begin{equation*}
  \mu < \min \Re\, \mathrm{spec} \,A(p),
\end{equation*}
where $\mu$ is possibly negative. However, we still know that
\begin{equation*}
  0 < \nu < \min \Re\, \mathrm{spec} \,\DD X|_p,
\end{equation*}
as $X$ has a strictly positive source at $p$. As argued before, both these estimates hold in a whole neighborhood of $(X, A) \in \mathscr{U}$.

The choice of coordinates induces a splitting
\begin{equation*}
  C^\infty(U, V) = \mathcal{P}_N \oplus \m^{N+1}.
\end{equation*}
With respect to this splitting, the operators $\mathbf{T}$ and $\mathbf{S}$ can be represented in matrix form by
\begin{equation*}
\mathbf{T} = 
\begin{pmatrix}
T_0 & 0\\
T_1 & \hat{\mathbf{T}}
\end{pmatrix}
 ~~~~~~~~\text{and}~~~~~~~~
 \mathbf{S} = 
\begin{pmatrix}
S_0 & 0\\
S_1 & \hat{\mathbf{S}}
\end{pmatrix}.
\end{equation*}
Clearly $S_0 = T_0^{-1}$, so $S_0$ depends smoothly on all data as $\mathcal{P}_N$ is finite-dimensional. We now claim that $\hat{\mathbf{S}}=\hat{\mathbf{S}}(X^\prime, A^\prime)$ is given by the same formula \eqref{SolutionFormula} as before. Formally, this is true for the same reason as before, but we need to show that $\hat{\mathbf{S}}(X^\prime, A^\prime)$ is actually well-defined on $\m^{N+1}$ and that it preserves this space. From Lemma \ref{FlatSolutionsEstimate}, we know that a function $v$ is in $\m^{N+1}$ if and only if it fulfills
\begin{equation*}
  |v(\Phi_t^{X^\prime}(y))| \leq C e^{\nu N t} |y|^N
\end{equation*}
for each $t \leq 0$ and all $y$ in a compact set $K \subseteq U$. Hence
\begin{align*}
 |\hat{\mathbf{S}}(X^\prime, A^\prime)v(y)| 
 \leq \int_{-\infty}^0 |E_{X^\prime, A^\prime}(t, y)^{-1}| |v(\Phi_t^{X^\prime}(y))| \dd t 
 \leq \int_{-\infty}^0 e^{(\mu + N\nu)t}|y|^N \dd t.
\end{align*}
This clearly converges if we choose $N$ big enough to have $\mu + N\nu > 0$. Furthermore, by the remark before, this estimate shows that $\hat{\mathbf{S}}(X^\prime, A^\prime)v \in \m^{N+1}$, at least if we show that it is smooth as well. However, for $v \in \m^{N+1}$, the same estimates as in the proof of \ref{LemmaSmoothDependence} can be carried out to obtain this result. 

It remains to show that $S_1$ is bounded, but clearly
\begin{equation*}
  S_1 = - \hat{\mathbf{S}}T_2 S_0,
\end{equation*}
so the result follows.
\end{proof}

\begin{remark}
For the case of non-trivial kernels, one can obtain a similar but more complicated result by introducing finite-dimensional auxiliary spaces, which take care of the kernel and cokernel (compare \cite[Thm.\ II.3.3.3]{hamilton} for elliptic operators).
\end{remark}


\section{Applications and further Discussion} \label{section4}

Choose coordinates as in the proof of Lemma \ref{LemmaEigenvalues}. We observe from the proof that in such coordinates, natural candidates for eigenfunctions are the polynomials $z^\alpha \cdot b_j$, where $b_j$ is an eigenvector of $A(0)$ with eigenvalue $\rho_j$. In the case that we can furthermore choose (probably non-Euclidean) coordinates such that we have
\begin{equation*}
  \DD_X + A = \sum_{j=1}^n \mu_j z_j \frac{\partial}{\partial z_j} + A(0),
\end{equation*}
these functions are exactly the eigenfunctions. However, in general it is not possible to linearize the vector field $X$ without imposing additional conditions on the eigenvalues $\mu_1, \dots, \mu_n$. For example, a linearization is always possible if for all $j = 1, \dots n$ and all $\alpha \in \N_0^n$ with $|\alpha|>1$, we have
\begin{equation}
 \mu_j \neq \alpha_1 \mu_1 + \dots + \alpha_n \mu_n.
\end{equation}
This is the Poincaré-Sternberg Linearization Theorem \cite{sternberg57}. Clearly, this is an open condition on the initial data in the Fréchet topology. In this sense, the "generic" vector field $X$ will possess a linearization and the "generic" operator $\nabla_X + A$ will admit $m(\lambda)$ eigenvectors to the eigenvalue $\lambda$ (compare Remark \ref{RemarkOnMlambda}).

We now give an example where the vector field fails to possess a smooth linearization.

\begin{example} \label{ExampleLowRegularity}
Remember Example \ref{JordanNormalFormOnP2}, where we looked at $X = \grad \phi$ with $\phi(y) = \frac{1}{2} y_1^2+ y_1^2y_2 + y_2^2$. By calculating the matrix representation of $[\DD_X]_2$ on $\mathcal{P}_2$, we saw that the equation
\begin{equation} \label{ExampleEquation}
  \DD_X u = \lambda u
\end{equation}
has the eigenvalue $2$ with "arithmetic multiplicity" $m(2)=2$, but that there is only one function $u$ solving \eqref{ExampleEquation} with $\lambda = 2$. We conclude that there are no smooth coordinates $\tilde{y}$ such that $X$ is linear, i.e.\ $X = \tilde{y}_1\partial_{\tilde{y}_1} + 2 \tilde{y}_2 \partial_{\tilde{y}_2}$, because otherwise, $\tilde{y}_1^2$ and $\tilde{y}_2$ would be smooth eigenfunctions for the eigenvalue $2$. 

By the Hartmann-Grobmann Theorem however (see e.g.\ \citep[p.\ 127]{perko91}), there exist such diagonalizing coordinates $\tilde{y}$ of regularity $C^1$. Hence, there exists another eigenfunction which is only $C^1$. This shows that the solution theory of the equation \eqref{TheEquation0} becomes quite different when considering solutions of lower regularity. This is a huge difference to the theory of elliptic PDE, as all eigenfunctions are automatically smooth in the elliptic case.
\end{example}

Let us now turn to the examples mentioned in the introduction.

\begin{example}[Heat Kernel Expansion] \label{ApplicationHeat}
 Let $L$ be a generalized Laplace type operator, acting on sections of a vector bundle $\V$ (equipped with a scalar product) over a compact Riemannian manifold $M$. There exists a unique connection $\nabla$ on $\V$ such that $L = -\mathrm{tr} \nabla^2 + K$ for some endomorphism field $K$.
 It is well-known that the solution semigroup $e^{-tL}$ to the heat equation $(\partial_t + L)u = 0$ has a smooth integral kernel $k_t$, which has an asymptotic expansion for $t \searrow 0$ of the form
\begin{equation} \label{HeatKernelApproximation}
  k_t(p, q) \sim (4 \pi t)^{-n/2} \exp\Bigl({ -\frac{1}{4t} d(p, q)^2}\Bigr)\sum_{j=0}^\infty t^j \Phi_j(p, q).
\end{equation}
Here, $d(p, q)$ is the Riemannian distance between $p$ and $q$ (for details of all this, see \cite[Ch.\ 2]{berlinegetzlervergne}.

An important role is played here by the Jacobian determinant of the exponential map,
\begin{equation*}
  \mu(q) := \mu(p, q) := | \det ( \dd_X \exp_p ) | ~~~~~ \text{where}~~~~~ q = \exp_p(X).
\end{equation*}
Set $\Psi_j(p, q) = \mu(p, q)^{1/2}\, \Phi_j(p, q)$. Then for each $p \in M$, the $\Psi_j$ fulfill the transport equations \citep[Prop.\ 2.24]{berlinegetzlervergne}
\begin{align}
 \nabla_X \Psi_0(p, \cdot) &= 0 \label{TransportZero}\\
 (\nabla_X + j)\Psi_j(p, \cdot) &= - B \Psi_{j-1}(p, \cdot), ~~~~ j=1, 2, \dots \label{TransportHeat}
\end{align}
in the bundle $\V^*_p \otimes \V$, in a geodesic neighborhood of $p$. Here, $B := \mu^{1/2} \circ L \circ \mu^{-1/2}$ while
\begin{equation*}
  X = x^1 \frac{\partial}{\partial x^1} + \cdots +  x^n \frac{\partial}{\partial x^n}
\end{equation*}
is the radial vector field around $p$ ($x$ being a geodesic normal chart around $p$). Of course, all entities in the equation depend on $p$.

By Thm.\ \ref{ThmEigenvalues}, the eigenvalues to the $j$-th transport equation are
\begin{equation*}
  \lambda = \alpha_1 + \dots + \alpha_n + j, ~~~~~~\alpha \in \N_0^n.
\end{equation*}
In particular, $\lambda = 0$ is an eigenvalue with multiplicity $m(\lambda) =1$ iff $j=0$ and not an eigenvalue if $j>0$. Hence by Thm.\ \ref{ThmFredholm}, for each initial condition for $\Psi_0(p, \cdot)$ at $p$ (in particular, $\Psi_0(p, p) = \id$), there is exactly one solution $\Psi_0$ of the equation with $j=0$, and all $\Psi_j$ for $j>0$ are already uniquely determined and well-defined (usually, one chooses $\Psi_0(p, p) = \mathrm{id}_{\V}|_p$.
 
The map that sends $p$ to the radial vector field $X$ centered at $p$ is smooth as a map $U \longrightarrow \X_U$ for any geodesically convex subset $U$ of $M$, so by Thm.\ \ref{SmoothDependence} the $\Psi_j$ are smooth sections of the bundle $\V^* \boxtimes \V$ over the base
\begin{equation*}
  M \bowtie M := \{(p, q) \mid \text{$p$ and $q$ are not conjugate}\}.
\end{equation*}
Indeed, for every $(p, q) \in M \bowtie M$, one can find a geodesically convex open subset $U$ of $M$ containing both points and such that $U \times U \subset\subset M \bowtie M$. Then the radial vector field arount ${p^\prime}$ is in $\mathrm{int}(\X_U)$ for each $p^\prime \in U$. Furthermore, by \ref{SmoothDependence}, the functions $\Psi_j$ depend smoothly on the metric $g$ and the operator $L$, because both the radial vector fields and the function $\mu$ depends smoothly on $g$.

From \eqref{TransportZero}, it follows that $\Psi_0(p, q) = \Pi(p, q)$, the parallel transport along $\gamma_{p, q}$, the unique shortest geodesic connecting $p$ to $q$ (this is well-defined because $(p, q) \in M \bowtie M$). The higher terms are given by the formula
\begin{equation}
\Psi_j(p, q) = - \Pi(p, q)\int_0^1 s^{j-1} \,\Pi(p, \gamma_{p, q}(s))^{-1} B\Psi_{j-1}(p, \gamma_{p, q}(s))\, \dd s.
\end{equation}
In terms of $\Phi_j = \mu^{-1/2}\Psi_j$, this becomes
\begin{equation}
\Phi_j(p, q) = - \Phi_0(p, q)\int_0^1 s^{j-1} \,\Phi_0(p, \gamma_{p, q}(s))^{-1} L\Phi_{j-1}(p, \gamma_{p, q}(s))\, \dd s.
\end{equation}
This can easily be verified by calculating the components of the formula \eqref{SolutionFormula} in this particular case and making the substitution $t \mapsto s := \ln(t)$ (compare \cite[Thm.\ 2.26]{berlinegetzlervergne}).

\medskip

The Hadamard coefficients to the fundamental solution to the wave equation $\Box u = 0$ on a Lorentzian manifold satisfy an equation formally similar to same equation \eqref{TransportHeat}, with only $L$ replaced by some generalized wave operator (see \citep[p.\ 39]{baerginouxpfaeffle}).

\medskip

Also semiclassical (i.e.\ $\hbar$-dependent) versions of the heat kernel expansion can be dealt with in this approach (compare \cite{baerpfaeffle}, \cite{ludewig}).
\end{example}

\begin{example}[WKB Expansion]
Let $M$ be a Riemannian manifold with Laplace-Beltrami operator $\Delta = \dd^* \dd = -\div \circ \grad$  and let $V$ be a function with a non-degenerate minimum at $p \in M$ and $V(p)=0$. Then the WKB construction produces formal eigenvectors and eigenvalues of the operator $\hbar^2 \Delta + V$ in the limit $\hbar \searrow 0$ (for more details, see for example \cite{simon83}, \cite{dimassisjoestrand99}). These are formal power series in $\hbar$ of the form
\begin{equation} \label{AsymptoticEigenvalues}
  \boldsymbol{u} = e^{-\phi/\hbar}\sum\nolimits_{j=0}^\infty \hbar^j a_j, ~~~~~~ \boldsymbol{\lambda} = \sum\nolimits_{j=0}^\infty \hbar^j \lambda_j
\end{equation}
where $\phi$ is a positive solution of the {\em eiconal equation} $V = |\dd \phi|^2$ with $\phi(p)=0$ and the $a_j$ are required to solve the recursive transport equations
\begin{equation} \label{WKBEquation}
(\partial_X - \Delta \phi)a_j =  \lambda_0a_j - \Delta a_{j-1} + \lambda_1 a_{j-1} + \dots + \lambda_j a_0, ~~~~ j=0, 1, 2, \dots
\end{equation}
where $X = {2 \grad \phi}$. Because $V$ had a non-degenerate local minimum at $p$, we have in suitable normal coordinates $x$ that
\begin{equation*}
  \phi = \frac{1}{2}\sum\nolimits_{j=1}^n \mu_j (x^j)^2 + \dots ~~~~~~ \text{hence} ~~~~~~ X = 2\sum\nolimits_{j=1}^n \mu_j x^j \frac{\partial}{\partial x^j} + \dots,
\end{equation*}
where the numbers $\mu_j^2$ are the eigenvalues of the Hessian of $V$ and the dots indicate terms of higher order in $x$. This shows that the vector field $X$ indeed has a strictly positive source at $p$, so we can apply the Thm.\ \ref{ThmFredholm}. Furthermore, we obtain
\begin{equation*}
  (\Delta \phi) (p) = -(\mu_1 + \dots + \mu_n),
\end{equation*}
so that our equation has the eigenvalues
\begin{equation*}
  \lambda_0 = (2 \alpha_1 + 1)\mu_1 + \dots + (2\alpha_n + 1) \mu_n.
\end{equation*}
If we choose $\lambda_0$ with $m(\lambda_0)=1$ (e.g.\ $\alpha_1 = \dots = \alpha_n = 0$), then the solution space to the equation
\begin{equation*}
  (\partial_X - \Delta \phi)a_0 = \lambda_0 a_0
\end{equation*}
is $1$-dimensional by Thm.\ \ref{ThmFredholm}. In this case, there is a unique choice for all $\lambda_j$, $j>0$ such that the equation \eqref{WKBEquation} is solvable. This is because for $N$ big enough, the spaces $\mathcal{P}_N$ (defined in \ref{NotationPolynomialSpaces}) split into
\begin{equation*}
  \mathcal{P}_N = \mathrm{im}([\partial_X - \Delta \phi]_N - \lambda_0) \oplus \R [a_0]_N,
\end{equation*}
so there is a unique $\lambda_j$ such that the right hand side of \eqref{WKBEquation} lies in $\ker ((\nabla_X + A)^\prime- \lambda_0)_\perp$. Then by Thm.\ \ref{ThmFredholm}, the solution space is a one-dimensional affine space with direction $\R[a_0]$. Hence the coefficients $a_j$ are uniquely determined up to multiples of $a_0$. In total, this means that the space of formal power series \eqref{AsymptoticEigenvalues} fulfilling \eqref{WKBEquation} form a one-dimensional vector space over the field $\R(\!(\hbar)\!)$ of formal Laurent series in $\hbar$.

For the "degenerate case", i.e.\ when $m(\lambda_0)>1$, it is not clear that one can solve the transport equations recursively as in the non-degenerate case, because
\begin{equation*}
\dim \mathrm{Coker} ([\partial_X - \Delta\Phi]_N - \lambda_0) > 1,
\end{equation*}
but as above, one only has one variable $\lambda_j$ to make adjustments. In fact it turns out that the problem may not be solvable as stated above; instead, half-integer powers of $\hbar$ may occur in the expansions \eqref{AsymptoticEigenvalues}. In this case, one should refer to other methods, for example spectral theory over the field $\C(\!(\hbar)\!)$ (see \cite{kleinschwarz}) or FBI transform (compare \citep[2.3.8]{helffer88}).
\end{example}


\appendix

\section{Estimates on Initial Value Problems}

In this appendix, we establish some estimates on the growth of the solutions to initial value problems.

Let $V$ be a finite-dimensional real or complex  vector space with a norm and let $A_0 \in \End(V)$ be a linear map on $V$. The unique matrix solution to the initial value problem
\begin{equation} \label{InitValueConst}
  \dot{E_0}(t) = A_0 E_0(t), ~~~~~~~~ E_0(0) = \id
\end{equation}
is given by the exponential $E_0(t) = e^{t A_0}$. We denote by 
\begin{equation*}
\ell := \ell(A_0) := \min \Re \,\mathrm{spec} (A_0)
\end{equation*}
the smallest real part of the eigenvalues of $A_0$. This is a continuous function of $A_0$ (see \cite[Thm.\ II.5.14]{kato95}). For each $\varepsilon>0$, we have
\begin{equation} \label{EstimateOnE0}
  \|E_0(t)E_0(s)^{-1}\| \leq M e^{(t-s)(\ell-\varepsilon)} ~~~~ \text{for} ~~ t \leq s \leq 0,
\end{equation}
where
\begin{equation} \label{DefinitionOfM}
  M = M(A_0, \varepsilon) := \sup_{t \leq 0} \| \exp\bigl\{t(A_0 - \ell(A_0)+\varepsilon)\bigr\}\|
\end{equation}
is a continuous function of $A_0$ and $\varepsilon$. Note that for each $\nu \in \R$, we have $M(A_0 + \nu, \varepsilon) = M(A_0, \varepsilon)$. 

For the more general case of non-constant $A$, we need the following lemma, which we adapt from \citep[Lemma 1.1]{OrdinaryDiffOps}. For convenience, we repeat the proof here, adapted to our situation, as it is only stated for positive time in the literature.

\begin{lemma} \label{LemmaPerturbation}
Let $E_0$ as above and let $E$ be the solution of the initial value problem
\begin{equation} \label{InitValueProb}
  \dot{E}(t) = A(t)E(t), ~~~~~~~~ E(0) = \id,
\end{equation}
where $A \in C((-\infty, 0], \End(V))$ is a continuous path of matrices. Then for each $\varepsilon >0$ and each $t \leq 0$, we have
\begin{align*}
  \|E(t)\| &\leq M \exp\left(t\bigl(\ell - \varepsilon - M \sup_{t \leq 0}\|A(t) - A_0\|\bigr)\right)
\end{align*}
where  $\ell = \ell(A_0)$ and $M=M(A_0, \varepsilon)$ were defined above.
\end{lemma}

\begin{proof}
First note that we may assume w.l.o.g.\ that 
\begin{equation*}
  M(A_0, \varepsilon)  \sup_{t \leq 0}\|A(t) - A_0\| = \ell(A_0) - \varepsilon,
\end{equation*}
because the left side is invariant under the change $A \rightsquigarrow A + \nu$, $A_0 \rightsquigarrow A_0 + \nu$, while the right side increases by $\nu$. Now fix $t_0<0$ and consider the operator on $C([t_0, 0], \End(V))$ defined by
\begin{equation*}
  TF(t) = E_0(t) -  E_0(t) \int_t^0 E_0(s)^{-1} (A(s) - A_0) F(s) \dd s, ~~~~~ t \in [t_0, 0]
\end{equation*}
It is easy to verify that $T$ is a contraction and maps the closed ball around zero of radius $M$ in $C([t_0, 0], \End(V))$ (with respect to the $C^0$-norm) to itself, while on the other hand, $E$ is a fixed point of $T$ (as is easily verified by direct calculation, the solution $E$ fulfills the integral equation $TE(t) = E(t)$). Now by Banach's Fixed Point Theorem, there is exactly one fixed point and the fixed point must lie in the Ball of radius $M$. Hence $\|E(t)\| \leq M$ for all $t\leq 0$, which is the proposition under our initial assumption.
\end{proof}

The above lemma shows how to estimate the solution of \eqref{InitValueProb} by comparing it with $A_0$. However, we are interested in an estimate with the {\em same} exponent as in \eqref{EstimateOnE0}. This is the content of the following theorem.

\begin{theorem} \label{ThmEstimateLemma}
Let $A \in C((-\infty, 0], \End(V))$, choose $\varepsilon > 0$ and assume that there exists $A_0 \in \End(V)$ and $t_0 \leq 0$ such that
\begin{equation*}
  \|A(t) - A_0\| < \frac{\varepsilon/2}{M(A_0, \varepsilon/2)} ~~~~~~ \text{for all}~~ t \leq t_0.
\end{equation*}
Then we have the estimate
\begin{equation*}
  \|E(t)\| < C e^{t(\ell(A_0)-\varepsilon)}.
\end{equation*}
for the solution of the initial value problem \eqref{InitValueProb}, where
\begin{equation*}
  C = M(A_0, \varepsilon/2) M(A_0, \varepsilon) \exp\left(-t_0 M(A_0, \varepsilon)\sup_{t \leq 0}\|A(t) - A_0\|\right)
\end{equation*}
\end{theorem}

\begin{proof}
For $t_0 \leq 0$ fixed, let $\tilde{E}$ be the solution of the initial value problem
\begin{equation*}
  \dot{\tilde{E}}(t) = A(t + t_0) \tilde{E}(t), ~~~~~~~~ \tilde{E}(0) = \id.
\end{equation*}
Then $E(t) = \tilde{E}(t - t_0){E}(t_0)$. From Lemma \ref{LemmaPerturbation}, we get
\begin{equation*}
  \| \tilde{E}(t)\| \leq M(A_0, \varepsilon/2) \exp\Bigl\{t\bigl(\ell(A_0) - \varepsilon/2 - M(A_0, \varepsilon/2) \|A(t+t_0) - A_0\|\bigr)\Bigr\} < M e^{t(\ell-\varepsilon)}.
\end{equation*}
Therefore,
\begin{align*}
  \|E(t)\| \leq \|\tilde{E}(t - t_0)\| \| {E}(t_0)\| < \| {E}(t_0) \| M e^{-t_0(\ell - \varepsilon)} e^{ t(\ell - \varepsilon )}.
\end{align*}
Finally, $\|E(t_0)\|$ can be estimated as in \ref{LemmaPerturbation}, giving
\begin{equation*}
  C = M(A_0, \varepsilon/2) M(A_0, \varepsilon) \exp\Bigl(-t_0 M(A_0, \varepsilon) \|A - A_0\|_{C_0}\Bigr).
\end{equation*}
\end{proof}

\begin{remark} \label{RemarkInverse}
One easily checks that the inverse $F(t):= E(t)^{-1}$ fulfills the problem
\begin{equation*}
  \dot{F}(t) = - F(t) A(t), ~~~~~~~~ F(0) = \id.
\end{equation*}
Under similar assumptions as in the above theorem, we obtain that
\begin{equation*}
  | F(t) w| \leq C e^{t ( \ell(-A_0) - \varepsilon )}|w|
\end{equation*}
for all $w \in V$ and hence, be replacing $w \rightsquigarrow E(t) w$, we get that
\begin{equation*}
  |E(t)w| \geq \frac{1}{C} e^{t ( - \ell(-A_0) + \varepsilon )}|w|
\end{equation*}
for all $w \in V$. In fact, $-\ell(-A_0) = \max \Re \, \mathrm{spec} \, A_0$.
\end{remark}

\bibliography{Literatur}

\end{document}